\documentclass[12pt,a4paper]{article}

\usepackage{amsmath,amssymb,amsfonts,amsthm}
\usepackage{a4wide}
\usepackage{graphicx}

\usepackage[colorlinks=true,citecolor=black,linkcolor=black,urlcolor=blue]{hyperref}

\usepackage{enumerate,color}

\theoremstyle{plain}
\newtheorem{theorem}{Theorem}[section]
\newtheorem{lemma}[theorem]{Lemma}
\newtheorem{corollary}[theorem]{Corollary}

\newtheorem{observation}[theorem]{Observation}

\theoremstyle{definition}
\newtheorem{definition}[theorem]{Definition}

\newtheorem{conjecture}[theorem]{Conjecture}

\newtheorem{question}[theorem]{Question}

\theoremstyle{remark}
\newtheorem{remark}[theorem]{Remark}

\newcommand{\comment}[1]{}
\newcommand{\rbg}{restricted frame graph}
\newcommand{\rbgs}{restricted frame graphs}
\newcommand{\Rbg}{Restricted frame graph}

\newcommand{\next}{\textsc{next}}
\newcommand{\gssp}{graph-stable set pair}
\newcommand{\add}{\textsc{add}}
\newcommand{\join}{\textsc{join}}

\newcommand{\F}{\mathcal F}
\newcommand{\D}{\mathcal D}
\newcommand{\R}{\mathbb R}
\renewcommand{\S}{\mathcal S}

\title{Restricted frame graphs and a conjecture of Scott}

\author{J\'er\'emie Chalopin\thanks{LIF, CNRS \&
  Univ. Aix-Marseille, Marseille, France. Partially supported by ANR Project MACARON
(\textsc{anr-13-js02-0002}).}
\and%
Louis Esperet\thanks{G-SCOP, CNRS \&
  Univ. Grenoble Alpes, Grenoble, France. Partially supported by ANR Project Heredia
  (\textsc{anr-10-jcjc-0204-01}), ANR Project Stint
  (\textsc{anr-13-bs02-0007}), and LabEx PERSYVAL-Lab
  (\textsc{anr-11-labx-0025-01}).}
\and%
Zhentao Li\thanks{D\'epartement d'Informatique,
   \'Ecole Normale Sup\'erieure, Paris, France}
\and%
Patrice Ossona de Mendez\thanks{CAMS, CNRS \&
\'Ecole des Hautes \'Etudes en Sciences Sociales, Paris, France, and
IUUK, Charles University, Prague, Czech Republic. Partially supported by grant ERCCZ LL-1201,  by the European Associated Laboratory ``Structures in
Combinatorics'' (LEA STRUCO), 
and by ANR Project Stint \textsc{anr-13-bs02-0007}.}
}

\begin{document}

\maketitle

%\sloppy

\begin{abstract}
Scott proved in 1997 that for any tree $T$, every graph with bounded
clique number which does not contain any subdivision of $T$ as an
induced subgraph has bounded chromatic number. Scott also conjectured
that the same should hold if $T$ is replaced by any graph $H$. Pawlik
{\it et al.}  recently constructed a family of triangle-free
intersection graphs of segments in the plane with unbounded chromatic
number (thereby disproving an old conjecture of Erd\H os). This shows
that Scott's conjecture is false whenever $H$ is obtained from a
non-planar graph by subdividing every edge at least once.

It remains interesting to decide which graphs $H$ satisfy Scott's
conjecture and which do not. In this paper, we study the construction
of Pawlik {\it et al.} in more details to extract more counterexamples
to Scott's conjecture. For example, we show that Scott's conjecture is
false for any graph obtained from $K_4$ by subdividing every edge at
least once.  We also prove that if $G$ is a 2-connected
multigraph with no vertex contained in every cycle of $G$, then any
graph obtained from $G$ by subdividing every edge at least twice is a
counterexample to Scott's conjecture.
\end{abstract}
\maketitle

\section{Introduction}

A class of graph is \emph{$\chi$-bounded} if there is a function $f$
such that every graph $G$ in the class satisfies $\chi(G) \le
f(\omega(G))$, where $\chi(G)$ is the chromatic
number and $\omega(G)$ is the clique number of $G$. It is well known
that the class of all graphs is not $\chi$-bounded~\cite{Mycielski}.

Gy\'arf\'as~\cite{gyarfas73} (see also~\cite{Gya87}) conjectured that for any tree $T$,
the class of graphs that do not contain $T$ as an induced subgraph is
$\chi$-bounded. This conjecture is still open, but Scott proved the
following topological variant in 1997 \cite{Sco97}: for any tree $T$,
the class of graphs that do not contain any subdivision of $T$ as an
induced subgraph is $\chi$-bounded. Scott conjectured that the same
property should hold whether $T$ is a tree or not. On the other hand,
it is easy to see that Gy\'arf\'as's conjecture is false if $T$ contains a cycle as there are graphs of arbitrarily high girth (ensuring no copy of $T$ appears) and high chromatic number~\cite{Erdos}.

\begin{conjecture}[Scott's conjecture \cite{Sco97}] 
  For any $H$, the class of graphs excluding all subdivisions of $H$
  as an induced subgraph is $\chi$-bounded.
\end{conjecture}

A \emph{$\ge \! \! k$-subdivision} of a (multi)graph $G$ is a
graph obtained from $G$ by subdividing each edge at least $k$ times,
i.e. replacing every edge of $G$ by a path on at least $k+1$ edges.  A
recent result \cite{pawlik2014} (see \cite{PKK13} for a follow-up) shows that Scott's conjecture is
false whenever $H$ is a $\ge \! \! 1$-subdivision of a non-planar
graph. The proof is based on a construction of a family of
triangle-free intersection graphs of segments in the plane with
unbounded chromatic number (the existence of such graphs also
disproved a conjecture of Erd\H os). Since no $\ge \! \! 1$-subdivision
of a non-planar graph can be represented as the intersection of
arcwise connected sets in the plane (in particular, such subdivisions
cannot be represented as an intersection graph of line segments) no
such graph appears as an induced subgraph in the
construction. Therefore, graphs in the construction exclude all
subdivisions of 1-subdivisions of non-planar graphs as induced
subgraphs.  Hence, the 1-subdivision of any non-planar graph is a
counterexample to Scott's conjecture. 

Recently, Walczak~\cite{Wal14} showed how to slightly modify the
construction of \cite{pawlik2014,PKK13} to obtain a family of graphs with
no stable sets of linear size (in particular, with unbounded
fractional chromatic number). Therefore, the following stronger
result can be deduced: for any non-planar graph $H$, there exist
graphs with no $\ge \! \! 1$-subdivision of $H$ as an induced
subgraph, and with no stable sets of linear size.

Note that the construction of \cite{pawlik2014,PKK13} gives the same family
of graphs as a construction of Burling \cite{Bur65}, who proved in 1965 that
triangle-free intersection graphs of axis-parallel boxes in $\R^3$
have unbounded chromatic number.

\subsection*{Our results}

Our original goal was to characterize \emph{all} graphs $H$ such that
no subdivision of $H$ appears as an induced subgraph in the
construction of \cite{pawlik2014, PKK13} (this extended set of graphs
would then provide new counterexamples to Scott's conjecture
\cite{Sco97}). Unfortunately, our characterization is incomplete but
we still provide new counterexamples to Scott's conjecture. On the
other hand, we are able to give a complete characterization of all graphs
$H$ that are a $\ge \! \! 2$-subdivision of some multigraph, and such
that no subdivision of $H$ appears as an induced subgraph in the
construction.

A consequence of Pawlik {\it et al.}'s result \cite{pawlik2014,PKK13}
is that Scott's conjecture is false for any graph obtained from $K_5$
by subdividing every edge at least once.  We show that Scott's
conjecture is also false for any graph obtained from $K_4$ by
subdividing every edge at least once. Note that proving that Scott's conjecture
holds for any subdivision of $K_3$ is equivalent to a long standing
conjecture of Gy\'arf\'as~\cite{Gya87}, which remains open.  We also
prove that if $G$ is a 2-connected multigraph with no vertex
intersecting every cycle of $G$, then any graph obtained from $G$ by
subdividing every edge at least twice is a counterexample to Scott's
conjecture. As our focus is on the construction, we do not prove Scott's
conjecture is true for any particular graph, only that it cannot be proven
false using the construction in some cases.

Our proof uses the following remarkable aspect of Pawlik {\it et
  al.}'s \cite{PKK13} proof (see also \cite{KPW13}): The
graphs in the construction can be obtained not only as intersection
graphs of segments in the plane, but also as intersection of a wide range
of arcwise connected shapes in the plane. In this paper, we will use
the fact that graphs in the construction can be represented as \rbg s
(see Section~\ref{sec:rbg} for the definition).

Instead of focusing on the construction, we will focus on
triangle-free graphs $H$ such that no subdivision of $H$ can be
represented as a \rbg. These subdivisions do not appear as induced
subgraphs in the construction, so it follows that such graphs $H$ are
counterexamples to Scott's conjecture. It turns out that graphs in the
modified construction of Walczak~\cite{Wal14} can be obtained from
graphs in Pawlik {\it et al.}'s construction \cite{pawlik2014,PKK13,KPW13} by
adding twins. As the class of restricted frame graphs is stable by the operation of twin addition (see Remark~\ref{rem:twin}),  the graphs
$H$ we find are also counterexamples to a weaker version of Scott's
conjecture, where the chromatic number is replaced by the fractional
chromatic number.

In Section~\ref{sec:tfreerbg}, we characterize connected triangle-free
graphs without \emph{full star cutsets} (defined in that section)
which are \rbg s. Among other consequences, this characterization
directly implies that Scott's conjecture is also false for any graph
$H$ obtained from $K_4$ by subdividing every edge at least once. It also
implies that Scott's conjecture is false for any graph $H$ which is a
$\ge \! \! 2$-subdivision of a 2-connected multigraph $G$ with no
vertex contained in all cycles. We investigate these $\ge \! \! 2$-subdivisions in
more details in Section~\ref{sec:2sub}, where we show that for every
multigraph $G$, either all $\ge \! \! 2$-subdivisions of $G$ can be
represented as \rbg s, or none of them can (and it can be determined
in linear time whether $G$ satisfies the former or the latter
property).

\smallskip

While it might seem restrictive to study graphs that cannot be
represented as \rbg s (instead of graphs that do not appear in the
construction), it can be proven that in the case of $\ge \! \!
2$-subdivisions of multigraphs, this is not restrictive at all: for
every multigraph $G$, if some $\ge \! \!  2$-subdivision $H$ of $G$
can be represented as a \rbg, then $H$ appears as an induced subgraph in
the construction of Pawlik {\it et al.}
\cite{pawlik2014,PKK13,KPW13} and in the modified construction of
Walczak~\cite{Wal14}. So the construction can be thought of as
\emph{universal} for $\ge \! \!  2$-subdivisions of \rbg s. Details
about the construction and this final result are given in
Appendix~\ref{sec:cons}.

\section{\Rbg s}\label{sec:rbg}

As stated in the introduction, our proof relies on the analysis of the
following class of graphs.

\begin{definition}
An \emph{axis-parallel box} in $\R^2$ is the
cartesian product of two intervals of $\R$.
A \emph{frame} is the boundary of an axis-parallel box in $\R^2$.
\end{definition}

\begin{definition}\label{def:rbg}
  A graph $G$ is a {\em \rbg} if it is the intersection graph of a
  family of frames with
  these restrictions (see Figure~\ref{fig:inter} for the only allowed intersection).

  \begin{enumerate}
  \item
    Corners of a frame do not coincide with any point of another frame,
  \item
    the left side of any frame does not intersect any other frame,
  \item
    if the right side of a frame intersects a second frame, this right
    side intersects both the top and bottom of this second frame, and
  \item
    if two frames have non-empty intersection, then no frame is
  (entirely) contained in the intersection of the regions bounded by
  the two frames.
  \end{enumerate}
\end{definition}

\begin{figure}[htbp]
\centering \includegraphics[scale=0.5]{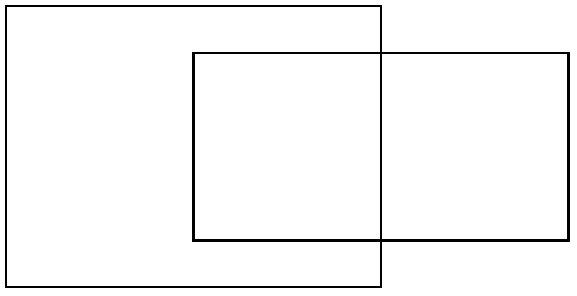}
\caption{The only possible intersection pattern between two frames in
  a \rbg.} \label{fig:inter}
\end{figure}

 A {\em representation} of a \rbg\ $G$ is a set of frames $\F=\{F_v\,|\,v
\in G\}$ where each $F_v$ is a frame and such that these frames
satisfy the above restrictions and $uv \in E(G)$ if and only if $F_u$
intersects $F_v$.

  We refer to $F_v$ as the {\em frame of $v$} and $v$ as the {\em
    vertex of $F_v$}.

\subsection*{Basic properties of \rbg s}

\begin{remark}
  As a consequence of restriction (1), when needed, we may assume all
  vertical sides (of all frames) occupy different $x$-coordinates and
  all horizontal sides (of all frames) occupy different
  $y$-coordinates.
\end{remark}

\begin{remark}
  A consequence of restriction (3) is that any frame which
  intersects the top edge of another frame also intersects the bottom
  edge of that frame.
\end{remark}

Figure~\ref{fig:inter} illustrates the only way two frames are allowed
to intersect in a representation of a \rbg. Note that by (4), no other
frame is allowed to be contained in the intersection of their
interior.

Two non-adjacent vertices $u$ and $v$ of a graph $G$ are said to be \emph{twins} if $N(u)=N(v)$.

\begin{remark}
\label{rem:twin}
If a graph $G$ has a restricted frame representation and if $v$ is a vertex of $G$, the graph obtained by adding a twin of $v$ to $G$  also has a restricted frame representation (where the twin is represented by a frame just inside the frame of $v$).
\end{remark}

\section{On triangle-free \rbg s}\label{sec:tfreerbg}

\subsection{Basic observations and a simple subclass of triangle-free \rbg s}

In this section, we characterize connected triangle-free \rbg s that cannot
be disconnected by the removal of the closed neighborhood of a
vertex. We first describe a simple subclass of \rbgs. 

%% We now wish to know which graphs are \rbgs and which are not as graphs
%% that are not \rbg s do not appear in the construction. We begin with a
%% simple subclass of \rbgs.

A graph obtained from a tree $T$ by adding a vertex $v$ adjacent to
every leaf of $T$ is called a \emph{chandelier}. The vertex $v$ is
called the \emph{pivot} of the chandelier. If the tree $T$ has the
property that the neighbor of each leaf has degree two, then the
chandelier is a \emph{luxury} chandelier. Note that any subdivision of
a (luxury) chandelier is a (luxury) chandelier.

\begin{lemma}\label{lem:trees}
  Any chandelier is a \rbg.
\end{lemma}

\begin{proof}
  Any rooted tree can be represented in such a way that the frame for the root has the leftmost left side, and the frame for leaves have the rightmost right side (see
  Figure~\ref{fig:btree}).  Now we simply add the pivot $v$ as a large box
  intersecting exactly these leaves frames.
\end{proof}

\begin{figure}[htbp]
\centering
\includegraphics[scale=1]{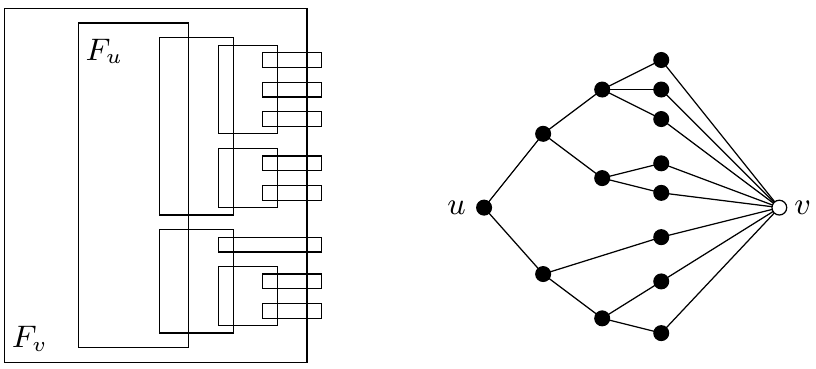}
\caption{A frame representation of a tree $T$ rooted in $u$ together
  with a vertex $v$ adjacent to each of the leaves of
  $T$.} \label{fig:btree}
\end{figure}

The remainder of this section is devoted to proving the converse of this
theorem for triangle-free graphs without full star-cutsets, which we
now define.

\medskip

A \emph{full star-cutset} in a connected graph $G$ is a set of
vertices $\{u\}\cup N(u)$ whose removal disconnects $G$. The
vertex $u$ is called the \emph{center} of the full star-cutset and the set
$\{u\}\cup N(u)$, denoted by $N[u]$, is called the \emph{closed
  neighborhood} of $u$.

\begin{observation}\label{obs-lux-path}
A tree $T$ has no full star-cutset if and only if $T$ is a path on at
most $4$ vertices. 

A chandelier has no full star-cutset if and only if it is a luxury
chandelier. 
\end{observation}

\begin{proof}
For any tree $T$, if $T$ contains a vertex $v$ of degree at least $3$,
any neighbor of $v$ is the center of a full star-cutset. Since any
path with length at least $4$ has a full star-cutset, $T$ is a path on
at most $4$ vertices.

It is easily checked that a luxury chandelier has no full star-cutset, by considering successively the cases where the deleted vertex is the vertex $v$, a leaf of the tree (this is where the assumption that the chandelier is luxury is used), or any other vertex of the tree.

In a chandelier that is not a luxury chandelier, some leaf $v$ of the
tree has a parent of degree at least 3 and is therefore the center of a full
star-cutset in the chandelier.
\end{proof}

We now state the main result of this section, whose proof will be given in Section~\ref{sec:pf}, after all needed lemmas.

\begin{theorem}\label{th:noscs}
  Suppose $H$ is a connected triangle-free graph with no full
star-cutset and $H^*$ is a subdivision of $H$.  Then $H^*$ is a
\rbg\ if and only if $H$ is either a path on at most 4 vertices or a
luxury chandelier.
\end{theorem}

This result has the following direct consequence.

\begin{corollary}\label{cor:noscs}
Every connected triangle-free graph $H$ with no full star-cutset which
is neither a path on at most 4 vertices, nor a luxury chandelier is a
counterexample to Scott's conjecture.
\end{corollary}

\begin{figure}[htbp]
\centering
\includegraphics[scale=0.8]{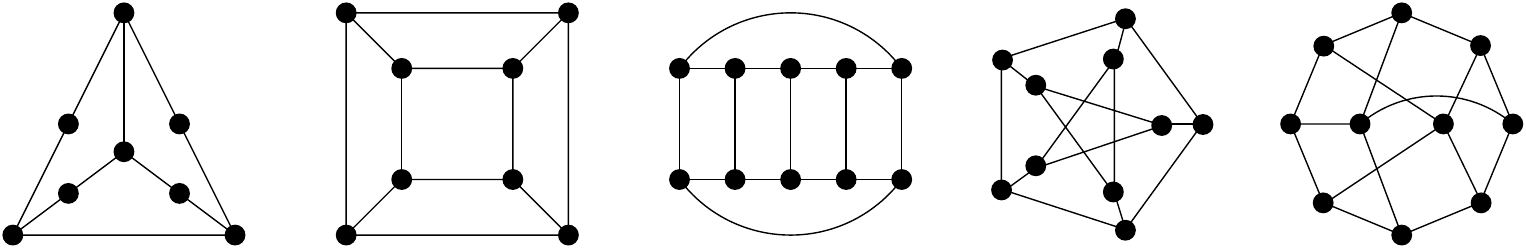}
\caption{Some small counterexamples to Scott's conjecture.} \label{fig:ex33}
\end{figure}

While the condition in Theorem~\ref{th:noscs} (and
Corollary~\ref{cor:noscs}) may seem technical, they still apply to a wide
range of graphs, including small ones. For example, all the graphs of
Figure~\ref{fig:ex33} (as well as their subdivisions) are counterexamples to
Scott's conjecture. The leftmost graph (the 1-subdivision of $K_4$), call
it $H$, is particularly interesting: any
triangle-free subdivision of $K_4$ which is not a subdivision of $H$
can be represented as a \rbg\ (see Appendix~\ref{sec:k4} for
details). On the other hand, such subdivisions have full
star-cutsets, so this shows that the technical condition is needed in
Theorem~\ref{th:noscs}. Extensions of Corollary~\ref{cor:noscs} remain open.

\smallskip

Note that cycles of length at least 5 are luxury chandeliers, so
Scott's conjecture for cycles of length at least 5, which is
equivalent to a conjecture of Gy\'arf\'as~\cite{Gya87}, remains open.

\subsection{Representations of \rbg s}

To prove Theorem~\ref{th:noscs}, we need some more properties of \rbg
s. We say that a frame $F_1$ {\em contains} a frame $F_2$ if $F_1$ and
$F_2$ do not intersect and $F_2$ is completely contained in the region
bounded by $F_1$.  If $F_1$ contains $F_2$, we also say that $F_2$ is
{\em inside} $F_1$.  If $F_1$ and $F_2$ do not intersect and $F_2$ is
not inside $F_1$, we say that $F_2$ is {\em outside} $F_1$.  Note that
 if two frames intersect,
none of these frames is inside another.

\begin{lemma}[Path Lemma]
  Let $\F=\{F_v\,|\,v\in G\}$ be a representation of a \rbg\ $G$. For
  any vertices $u,v,w \in G$ where $F_u$ is inside $F_v$ and $F_w$ is
  outside $F_v$ and any path $P$ from $u$ to $w$, $P$ either contains
  $v$ or has an edge to $v$.
\end{lemma}

\begin{proof}
  Assume that $P$ does not contain $v$. Since $F_u$ is inside $F_v$
  and $F_w$ is outside $F_v$, there are two consecutive vertices
  $x,y$ in $P$  such that $F_x$ is inside $F_v$ and $F_y$ is not inside
  $F_v$. Since $F_x$ and $F_y$ intersect, $F_y$ contains a curve
  connecting a point inside $F_v$ and a point outside $F_v$. It
  follows from Jordan's curve theorem that $F_y$ intersects
  $F_v$. Hence, $y$ is adjacent to $v$ in $G$.
\end{proof}

\begin{corollary}[Path Corollary]\label{pathcor}
  For any two frames $F_u, F_v$ in a representation $\F$ of a
  \rbg\ $G$ with $F_u$ inside $F_v$, all frames of vertices in the
  connected component of $G-N[v]$ containing $u$ are inside $F_v$.
\end{corollary}

The Path Lemma has the following important consequence.

\begin{lemma}[Cycle Lemma]\label{lem-cycle}
Let $\F=\{F_v\,|\,v\in G\}$ be a representation of a \rbg\ $G$. For
any induced cycle $C$ of $G$, there is a vertex $v \in C$ such that
$F_v$ contains the frame of every vertex in $V(C) - N[v]$.

Moreover, if there exists another vertex $u \in C$ such that $F_u$
contains the frame of every vertex in $V(C) - N[u]$, then $u$ is a
neighbor of $v$.
\end{lemma}

\begin{proof}
We may assume that $C$ contains at least 4 vertices, since otherwise
the result holds trivially for any vertex $v \in C$. We now show how
to find $v$. Among the subset $X_C$ of vertices of $C$ whose frame's
right edge intersects the frame of another vertex of $C$, pick $v \in
X_C$ with the largest $x$-coordinate for the right edge of its frame
$F_v$.

We claim that the frames of the vertices of $V(C) - N[v]$ are inside
$F_v$.  To see this, let $u\in C$ be a vertex whose frame intersects
the right edge of $u$, and let $w$ be the neighbor of $u$ on $C$ distinct
from $v$. By the maximality of $v$, the intersection of $F_w$ and
$F_u$ lies inside the region bounded by $F_v$ (otherwise we would have
chosen $u$ or $w$ instead of $v$). Since $C$ contains at least 4
vertices, $w$ is not adjacent to $v$ and therefore $F_w$ is inside
$F_v$. Since all elements of $V(C)-N[v]$ are in the same connected
component of $G-N[v]$ as $w$, the claim follows directly from
Corollary \ref{pathcor}.

Assume now that there exists another vertex $u \in C$ such that $F_u$
contains the frame of every vertex in $V(C) - N[u]$.  If $u$ and $v$
are not adjacent, then $F_u$ is inside $F_v$ and $F_v$ is inside
$F_u$, a contradiction.
\end{proof}

Given some representation $\F$ of $G$ as a \rbg\ and some induced cycle $C$ of $G$, we refer to a vertex
$v$ of $C$ whose frame contains every other frame of vertices of $C$
 non-adjacent to $v$ as a {\em big vertex of $C$}. The frame
$F_v$ is called a {\em big frame of $C$}.  The \emph{big vertices} of
$\F$ is the set of all big vertices for all cycles of $G$. Note that
these definitions depend on the chosen representation.  

\subsection{\texorpdfstring{$\ge \! \! 2$}{>=2}-Subdivisions of \texorpdfstring{$K_4$}{K4}}

To give a flavor of the proof of Theorem~\ref{th:noscs}, we show that
no $\ge \! \! 2$-subdivision of $K_4$ is a \rbg. Note that Theorem~\ref{th:noscs} implies the
stronger result that in fact, no $\ge \! \! 1$-subdivision of $K_4$ is a \rbg.

\begin{theorem}
No $\ge \! \! 2$-subdivision of $K_4$ is a \rbg.
\end{theorem}

\begin{proof}
Suppose not and let $\F$ be a representation of a $\ge \! \!
2$-subdivision $G$ of $K_4$. By the Cycle Lemma, any subdivided cycle
$C_1$ in $G$ has a big vertex $v_1$. Observe that since $G$ is a
$\ge \! \! 2$-subdivision of $K_4$, $G-N[v_1]$ contains a cycle $C_2$
(with a big vertex $v_2$),  is connected, and
contains some vertex of $C_1$. By the Path Corollary, it follows that
$F_{v_1}$ contains the frame of all vertices of $G-N[v_1]$, including
$v_2$. By symmetry, $F_{v_2}$ also contains $F_{v_1}$, a
contradiction.
\end{proof}

\subsection{Triangle-free \rbg s}

Recall that the construction of \cite{pawlik2014,PKK13,KPW13} has no triangle,
and we try to give a precise characterization of graphs which appear
as induced subgraphs in the construction. Graphs that contain a
triangle are clearly not contained in the construction, so we can
restrict ourselves to the study of triangle-free \rbg s.

Another point
is that the disjoint union of two graphs $G$ and $H$ is a \rbg\ if and
only if $G$ and $H$ are \rbg s. (As we shall see, induced subgraphs of
the construction are also closed under taking disjoint union.)  So we
can also restrict ourselves to the study of connected graphs.

\begin{lemma}\label{lem:bigv}
Let $\F=\{F_v\,|\,v\in G\}$ be a representation of a connected
triangle-free \rbg\ $G$. If $G$ has no full star-cutset, the big vertices
of $\F$ form a clique of $G$. In particular, there are at most two big
vertices.
\end{lemma}

\begin{proof}
Note that by Lemma~\ref{lem-cycle}, if a cycle $C$ has two big
vertices in $\F$, they are adjacent. Consider now two distinct cycles
$C_u$ and $C_v$ of $G$ and let $u$ be a big vertex of
$C_u$ in $\F$ and $v$ a big vertex of $C_v$ in $\F$, with $u \ne v$.  Assume for the sake of
contradiction that $u$ and $v$ are not adjacent. Since $G-N[v]$ is
connected, $u$ and some (remaining) vertex of $C_v$ are in the same
connected component of $G-N[v]$ (since $G$ is triangle-free, $C_v$ has
length at least four and so $V(C_v)-N[v]$ is non-empty). By the Path
Corollary, $F_v$ contains $F_u$. A symmetric argument yields that
$F_v$ is also contained in $F_u$, which is a contradiction.
\end{proof}

We conclude this subsection with two easy observations on
triangle-free graphs with no full star-cutset.

\begin{observation}\label{obs:fsct}
Let $G$ be a connected triangle-free graph with no full
star-cutset. Then $G$ has no cut-vertex of degree at least
three. Moreover if $G$ is not a path with at most 4 vertices, then $G$
has minimum degree at least 2.
\end{observation}

\begin{proof}
Assume that $v$ is a cut-vertex of degree at least 3. If $G-v$ has
more than two components, then removing the closed neighbourhood of
any neighbor $u$ of $v$ removes $v$ and at most one component of $G-v$
and so $u$ is the center of a full star-cutset, a
contradiction. Otherwise $G- v$ has precisely two components, and one
of the two components contains at least two neighbors $u$ and $w$ of
$v$, while the other contains at least one neighbor $t$ of $v$.  Since
$G$ is triangle-free, $u$ and $w$ are not adjacent, and so removing
$N[u]$ removes $v$ and disconnects $w$ from $t$. Thus, $N[u]$ is a
full star-cutset, which is a contradiction.

So $G$ contains no cut-vertex of degree at least 3.

Now assume that $G$ is not a path with at most four vertices. If $G$
contains a vertex $x$ of degree one, let $P$ be a maximal induced path
starting at $x$ and such that all internal vertices have degree
two. Let $y$ be the other end of $P$. Since $G$ is neither of path
with at most 4 vertices, nor a path with more than 4 vertices (such
graphs contain a full star-cutset), the vertex $y$ has degree at least
3. Moreover, $y$ is a cut-vertex separating $P$ from its other
neighbors, a contradiction to $G$ having no such vertex.
\end{proof}

\begin{observation}\label{obs:subdv_fsct}
If $G$ is a connected triangle-free graph with no full star-cutset and
$G$ is not a path on at most 4 vertices, then any subdivision $G^*$ of
$G$ is also a (connected triangle-free) graph with no full
star-cutset.
\end{observation}

\begin{proof}
Using induction, it is enough to prove this claim when $G^*$ is
obtained from $G$ by subdividing some edge $uv$ once, adding a new
vertex $w$. Note that for any non-neighbor $x$ of $u$ in $G$ (distinct from $u$), if
$N[x]$ is a cutset in $G^*$, then $N[x]$ is a cutset in $G$. Indeed,
the only difference between the connected components is that in $G^*-
N[x]$, $w$ is added to the connected component of $u$ in $G - N[x]$.
It follows that non-neighbors of $u$ (and by symmetry, non-neighbors
of $v$) in $G$ are not centers of star-cutsets in $G^*$. Since $G$ is
triangle-free, $u$ and $v$ have no common neighbors in $G$. Therefore, it
only remains to check that $u,v$ and $w$ are not centers of full
star-cutsets in $G^*$.

If $G^*-N[u]$ is disconnected then since $G-N[u]$ is connected,
$G^*-N[u]$ has a component containing only $v$. But since $u$ and $v$
have no common neighbor in $G$, it follows that $v$ has degree one in
$G$, which contradicts Observation~\ref{obs:fsct}. Hence, $u$ (and by
symmetry, $v$) is not the center of a full star-cutset in $G^*$.

Finally, if $G^*-N[w]=G-\{u,v\}$ is disconnected then since $G-N[u]$
and $G-N[v]$ are connected, there is no vertex outside $N[u]\cup
N[v]$. Since $G$ is triangle-free, it follows that $N(u)-v$, $N(v)-u$,
$\{u,v\}$ form a partition of $V(G)$, and since neither $N[u]$, nor
$N[v]$ is a cutset in $G$, each of these sets induces a connected
subgraph in $G$. Since $G-\{u,v\}$ is disconnected, there is no edge
between $N(u)-v$ and $N(v)-u$ in $G$. Therefore, $u$ and $v$ are
cut-vertices in $G$ and so by Observation~\ref{obs:fsct} they have
degree at most two. Hence, $G$ is a path of length at most 4, which is
a contradiction.
\end{proof}

\subsection{Proof of Theorem~\ref{th:noscs}}
\label{sec:pf}
We are now ready to prove Theorem~\ref{th:noscs}. To simplify the
presentation, we will prove the following lemma.

\begin{lemma}\label{lem:noscs}
Assume that $H$ is a connected triangle-free graph with no full
star-cutset.  If $H$ is a \rbg\ then $H$ is either a path or a chandelier.
\end{lemma}

We first prove our theorem assuming this lemma.

\medskip

\begin{proof}[Proof of Theorem~\ref{th:noscs}.]
Let $H$ be a connected triangle-free graph with no full star-cutset,
and let $H^*$ be a subdivision of $H$.

If $H$ is a path or a chandelier, then $H^*$ is also a
path or a chandelier, and it follows from Lemma~\ref{lem:trees} that
$H^*$ is a \rbg.

Conversely, suppose $H^*$ is a \rbg. We may assume $H$ is not a path
on at most 4 vertices or we already have the desired conclusion.  By
Observation~\ref{obs:subdv_fsct}, $H^*$ is also a connected
triangle-free graph with no full star-cutset. So by
Lemma~\ref{lem:noscs}, $H^*$ is a path or a chandelier. It follows
that $H$ is also a path or a chandelier. Since $H$ has no full
star-cutset, by Observation~\ref{obs-lux-path}, $H$ is a path on at
most 4 vertices or a luxury chandelier.
\end{proof}

It remains to prove Lemma~\ref{lem:noscs}.

\begin{proof}[Proof of Lemma~\ref{lem:noscs}.] Let $H$ be a connected
triangle-free graph with no full star-cutset that is not a
path. Assume that $H$ is a \rbg, and let $\F=\{F_v\,|\,v\in H\}$ be a
representation of $H$. By Lemma~\ref{lem:bigv}, $H$ has at most two
big vertices in $\F$. By Observation~\ref{obs:fsct}, $H$ has minimum
degree at least two. It follows that $H$ contains a cycle and
therefore $H$ has at least one big vertex in $\F$.

If there is exactly one big vertex $u$ in $H$, then $H-\{u\}$ is a
forest, and since $N[u]$ is not a cutset of $H$, $H-N[u]$ is a tree
$T'$. Observe that every neighbor $v$ of $u$ has exactly one neighbor
in $T'$, for if $v$ had two neighbors $v_1,v_2$ in $T'$, then the path
between $v_1$ and $v_2$ in $T'$ together with the vertex $v$ would
form a cycle in $H$ not containing $u$, contradicting the fact
$H-\{u\}$ is a forest. As $H$ is triangle-free, it follows that $H-u$ is a tree $T$, and as $H$ has minimum degree $2$, the leaves of $T$ are exactly the neighbors of $u$.
 This proves that $H$ is a chandelier, as desired.

We may now assume $H$ has precisely two big vertices $u$ and $v$ in
$\F$. Then $u$ and $v$ are adjacent by Lemma~\ref{lem:bigv}.  Since
the cycle for which $u$ is big has length at least 4, $F_u$ contains
the frame of some vertex on that cycle which is not adjacent to
$u$. Since $H$ has no full star cutset, $H-N[u]$ is connected and by
the Path Corollary, the frames of non-neighbors of $u$ are inside
$F_u$ and by symmetry the frames of non-neighbors of $v$ are inside
$F_v$.  Since $F_u$ and $F_v$ intersect, by property (4) of \rbg s, a
vertex cannot be non-adjacent to both $u$ and $v$. Since $H$ is
triangle-free, it implies that $\{u,v\}$, $N(u) - v$, $N(v) - u$ form
a partition of $V(H)$.  Since $N[v]$ is not a cutset, the subgraph
induced by the neighbors of $u$ distinct from $v$ is connected.  Since
$N[u]$ is not a cutset, the subgraph induced by the neighbors of $v$
distinct from $u$ is connected. Since $H$ is triangle-free, each of
these subgraphs is either empty or a single vertex.  Since $H$ is not
a path, it follows that $H$ is a cycle of length $4$, and thus $H$ is
a chandelier.
\end{proof}

\section{\texorpdfstring{$\ge \! \! 2$}{>=2}-Subdivisions of multigraphs}\label{sec:2sub}

 A vertex $v$ of a graph $G$ such that $G-v$ is a forest is called a {\em
    feedback vertex} of $G$. We make the following remark.

\begin{remark}
  If a full star-cutset in a $\ge \! \! 2$-subdivision of some multigraph $G$
  contains a vertex $v$ of $G$, then $v$ is a cut-vertex of $G$.
\end{remark}

Therefore, a direct
  consequence of Theorem~\ref{th:noscs} is the following:

\begin{corollary}\label{cor:2subdv}
Let $G$ be a 2-connected multigraph with no feedback vertex. Then no
$\ge \! \! 2$-subdivision of $G$ is a \rbg.
\end{corollary}

For example,
no $\ge \! \! 2$-subdivision of one of the graphs in Figure~\ref{fig:minors}
is a \rbg .

\begin{figure}[htbp]
\centering \includegraphics[scale=1]{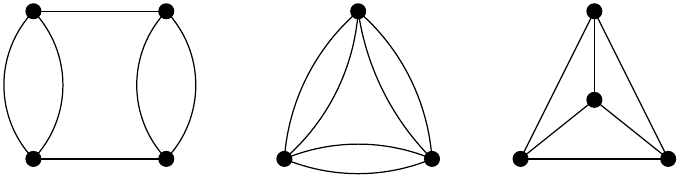}
\caption{Scott's conjecture is false when $H$ is a $\ge \! \!
  2$-subdivision of one of these graphs.} \label{fig:minors}
\end{figure}

We now characterize all \rbg s that are a $\ge \! \! 2$-subdivisions of some multigraphs.

\begin{lemma}\label{lem:glueing}
  Consider a \rbg~ $G_1$ and a chandelier $G_2$, and let $v \in
  G_1$. Then the graph $G$ obtained from the disjoint union of $G_1,
  G_2$ by identifying $v$ with the pivot of the chandelier $G_2$ is a
  \rbg.
\end{lemma}

\begin{proof}
  Let $\F=\{F_u\,|\,u\in G_1\}$ be a representation of $G_1$ as a
  \rbg. By the definition of a \rbg, there exists a small rectangular
  region $R$, whose interior contains the top right corner of $F_v$
  and which does not intersect or is contained in any other $F_u$ intersecting $F_v$ for $u\ne v$.

  Take a representation $\D$ of $G_2$ minus its pivot as a tree with all leaves on
  the right, such as the one depicted in Figure~\ref{fig:btree},
  shrink it and put $\D$ inside $R$, so that all leaves intersect the
  right side of $F_v$.

  The union of $\F$ and the shrunk version of $\D$ is a representation
  of $G$.
\end{proof}

\begin{figure}[htbp]
\centering
\includegraphics[scale=1]{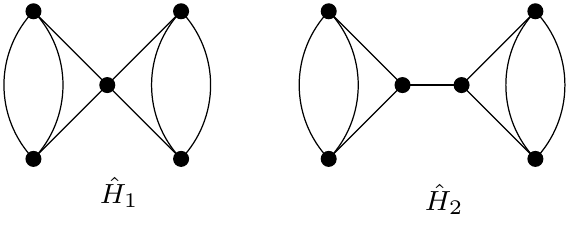}
\caption{Graphs $\hat H_1$ and $\hat H_2$.} \label{fig:bcv}
\end{figure}

\begin{lemma}\label{lem:specialgraphs}
  No $\ge \! \! 2$-subdivision $H^*$ of $\hat H_1$ or $\hat H_2$ (see
  Figure~\ref{fig:bcv}) is a \rbg.
\end{lemma}

\begin{proof}
  Suppose not and let $\F$ be a representation of $H^*$. Both $\hat H_1$
  and $\hat H_2$ have two disjoint digons which correspond to two vertex
  disjoint cycles $C_1,C_2$ in $H^*$. Since no vertex of these digons is
  a cut-vertex (and no edge of the digon is a cut-edge), no vertex in
  $C_i$ is the center of a star cutset in $H^*$. In particular $v_i$,
  the big vertex of $C_i$ in $\F$, is not the center of a star-cutset
  in $H^*$. Thus, by the Path Corollary, $F_{v_1}$ is inside $F_{v_2}$
  and $F_{v_2}$ is inside $F_{v_1}$, a contradiction.
\end{proof}

Given a connected graph $G$, define a bipartite graph $B_G=(U,V,E)$ as
follows. The elements of $U$ correspond to cut-vertices of $G$, while
the elements of $V$ correspond to maximal 2-connected components (also
called \emph{blocks\footnote{Note that a block might
consist only of two vertices of $G$ joined by an edge.}} in the
remainder) of $G$.  There is an edge in $B_G$
between an element of $U$ and an element of $V$ if the corresponding
cut-vertex belongs to the corresponding block. It is well known that
$B_G$ is a tree, called the \emph{block decomposition}, or the
\emph{block tree} of $G$, and that all leaves of the tree are in
$V$. If one block of $G$ is set as the root of the decomposition, we
obtain a \emph{rooted block decomposition} of $G$. In a rooted block
decomposition of $G$, the \emph{parent cut-vertex} of a block of $G$
distinct from the root is defined naturally.

\smallskip

Corollary~\ref{cor:2subdv} and Lemmas~\ref{lem:glueing}
and~\ref{lem:specialgraphs} have the following consequence.

\begin{theorem}\label{thm:alg}
For any connected multigraph $G$, either all $\ge \! \! 2$-subdivisions of $G$ are \rbg
s, or none of them are. Moreover, given a connected multigraph $G$, it can be
determined in linear time whether $G$ satisfies the former or the
latter property.
\end{theorem}

\begin{proof}

Our algorithm proceeds as follows.

\begin{enumerate}
\item
  Build the block tree of $G$. Remove leaves as long as their parent
  cut-vertex is a feedback vertex of the leaf.
\item
 If more than one block is left: answer {\bf no}.
\item
  If one block is left: decide if it has a feedback vertex and answer
  {\bf yes} if there is one and {\bf no} otherwise. To do this, find
  any cycle (greedily using DFS), greedily find an ear of the cycle.
  Then check if any of the two vertices at the end of the ear are
  feedback vertices.
\end{enumerate}

We now prove the correctness of the algorithm. It is enough to prove
that if the algorithm answers {\bf yes} for some input graph $G$, then
any $\ge \! \! 2$-subdivision of $G$ is a \rbg, while if the algorithm
answers {\bf no}, then no $\ge \! \! 2$-subdivision of $G$ is  a
\rbg.
 
Assume first that the algorithm answers {\bf yes}. Observe that any
$\ge \! \! 2$-subdivision of a 2-connected multigraph with a feedback
vertex $v$ is a chandelier with pivot $v$, or a path (if the
multigraph is a $K_2$). It follows from the algorithm that any $\ge \! \!
2$-subdivision of $G$ can be obtained from a chandelier or a path by
repeatedly applying the operation from Lemma~\ref{lem:glueing} or
adding a pendant vertex. This lemma implies that any $\ge \! \!
2$-subdivision of $G$ is a \rbg.

Assume now that the algorithm answers {\bf no}. So the algorithm
either stopped at step (2) or (3). If the algorithm stopped at step
(3), then by Corollary~\ref{cor:2subdv} no $\ge \! \! 2$-subdivision of
$G$ is a \rbg. So suppose the algorithm stopped at step (2). It
follows from step (1) that in this case, the decomposition contains at
least two leaves $G_1$ and $G_2$. Let $v_1$ and $v_2$ be the parent
cut-vertices of $G_1$ and $G_2$. Since $v_i$ is not a feedback vertex
of $G_i$, $G_i-v_i$ contains an (induced) cycle $C_i$. Moreover, since
$G_i$ is 2-connected (and distinct from a single edge, since otherwise
$v_i$ would be a feedback vertex of $G_i$), for any two
vertices on $C_i$ there are internally vertex disjoint paths
connecting them to $v_i$. Choose $C_i$ and these two vertices on $C_i$
in such a way that the maximum of the lengths of the two paths is
minimized. Then add a shortest path between $v_1$ and $v_2$ in $G$ to
$C_1$, $C_2$, and the four paths. It can be checked the subgraph
induced by the vertices of these paths and cycles is isomorphic to a
subdivision of either $\hat H_1$ or $\hat H_2$. It follows from
Lemma~\ref{lem:specialgraphs} that no $\ge \! \! 2$-subdivision of $G$ is
a \rbg.
\end{proof}

\section{Conclusion}

It was already known that any $\ge \! \! 1$-subdivision of a non-planar graph
is a counterexample to Scott's conjecture. This recent obervation was based
on the fact that a particular class of triangle-free graphs of
unbounded chromatic number can be represented as intersection graphs
of line segments in the plane. In this paper, we used the fact that
this particular class of graphs can be represented by intersection graphs of even
more specific objects in the plane, in order to provide a larger class of counterexamples. In
particular, we proved that any $\ge \! \!  2$-subdivision of a
2-connected multigraph is a counterexample to Scott's conjecture.

This was done without studying the construction itself, only the class
of intersection graphs containing it. A natural question is
whether studying the construction itself would provide a larger
class of couterexamples. We can show the answer is negative
when we restrict ourselves to $\ge \! \!  2$-subdivisions of
multigraphs. More details about this, as well as a description of the original construction, are given in Appendix~\ref{sec:cons}.

\medskip

For a given graph $H$, let $\mbox{Forb}^*(H)$ denote the class of graphs excluding all subdivisions of $H$
  as induced subgraphs. Many special cases of the following natural refinement of Scott's conjecture remains. 

\begin{question}\label{question:refinedscott}
  For which graphs $H$ is $\mbox{Forb}^*(H)$ $\chi$-bounded?
\end{question}

In view of Theorem~\ref{th:noscs} and
Corollary~\ref{cor:noscs}, one such refinement seems natural.

\begin{question}
  Is it true that for any luxury chandelier $G$, $\mbox{Forb}^*(G)$ is $\chi$-bounded?
\end{question}

\smallskip

Note that cycles of length at least 5 are luxury chandeliers, so a
positive answer to this question would imply that the class of graphs
with no induced cycles on at least 5 vertices is $\chi$-bounded, which is a
long-standing conjecture of Gy\'arf\'as~\cite{Gya87}.

\smallskip

Another special case prompted by Corollary \ref{cor:2subdv} is the following.

\begin{question}
  For which graphs $G$ do we have that for all subdivisions $H$ of $G$, $\mbox{Forb}^*(H)$ is $\chi$-bounded?
\end{question}

This we hope may have a fairly clean answer (as opposed to Question \ref{question:refinedscott}). It does not however escape the difficulty of the long-standing conjecture of Gy\'arf\'as~\cite{Gya87} as we can pick $H$ to be a long cycle.

\subsection*{Recent development}

After the submission of this article, there have been many exciting development.
Scott and Seymour \cite{CSSI} have answered (in the positive) the long standing question of Gy\'arf\'as~\cite{Gya87} stated as open here. Chudnovsky, Scott and Seymour~\cite{CSSV} have also proved a partial converse of the result in this article.

\section*{Acknowledgements} The work at the origin of this paper
was started during a workshop on $\chi$-bounded classes organized in
Lyon, France, in March 2012. The authors would like to thank the
organizers and participants of this meeting, in particular Feri
Kardo\v s, Fr\'ed\'eric Maffray, and St\'ephan Thomass\'e, for the
discussions initiating this work. We would also like to thank the anonymous referees for their many helpful suggestions.

\bibliographystyle{plain}

%%%%%%%%%%%%%%%%%%

%%%%%%%%%%%%%%%%%%

\clearpage

\appendix

\section{Subdivisions of \texorpdfstring{$K_4$}{K4}}\label{sec:k4}

It was proved by Scott
(see~\cite{LMT12}) that there exists a constant $c$ such that graphs
with no induced subdivisions of $K_4$ have chromatic number at most
$c$. It remains interesting to understand which subdivisions of $K_4$
are responsible for this bound on the chromatic number. 

The \emph{type} of a subdivision of $K_4$ is the number of
subdivided edges in the original copy of $K_4$. For instance, a
type 6 subdivision of $K_4$ is obtained from $K_4$ by subdividing each
edge at least once, while a type 0 subdivision of $K_4$ is just a copy
of $K_4$.

Corollary~\ref{cor:noscs} directly implies that any type 6 or type 5
subdivision of $K_4$, and any type 4 subdivision of $K_4$ in which the
non-subdivided edges of $K_4$ do not share a vertex are counterexamples to
Scott's conjecture. In this section, we show that every other
triangle-free subdivision of $K_4$ can be represented as a \rbg.

\smallskip

In Figure~\ref{fig:pathinsertion}, we show that starting from a frame
representation of a graph $G$, and given a specific edge $uv$ of $G$, we can
inductively construct frame representations of graphs obtained from
 $G$ by subdividing the edge $uv$. These operations are only valid if
the intersection of the frames of $u$ and $v$ in the original
representation of $G$ is of a certain type. We omit the details, since
it can be easily checked that these operations work fine in the
representations of the graphs of Figure~\ref{fig:repk4}.

\begin{figure}[htbp]
\centering \includegraphics[scale=0.55]{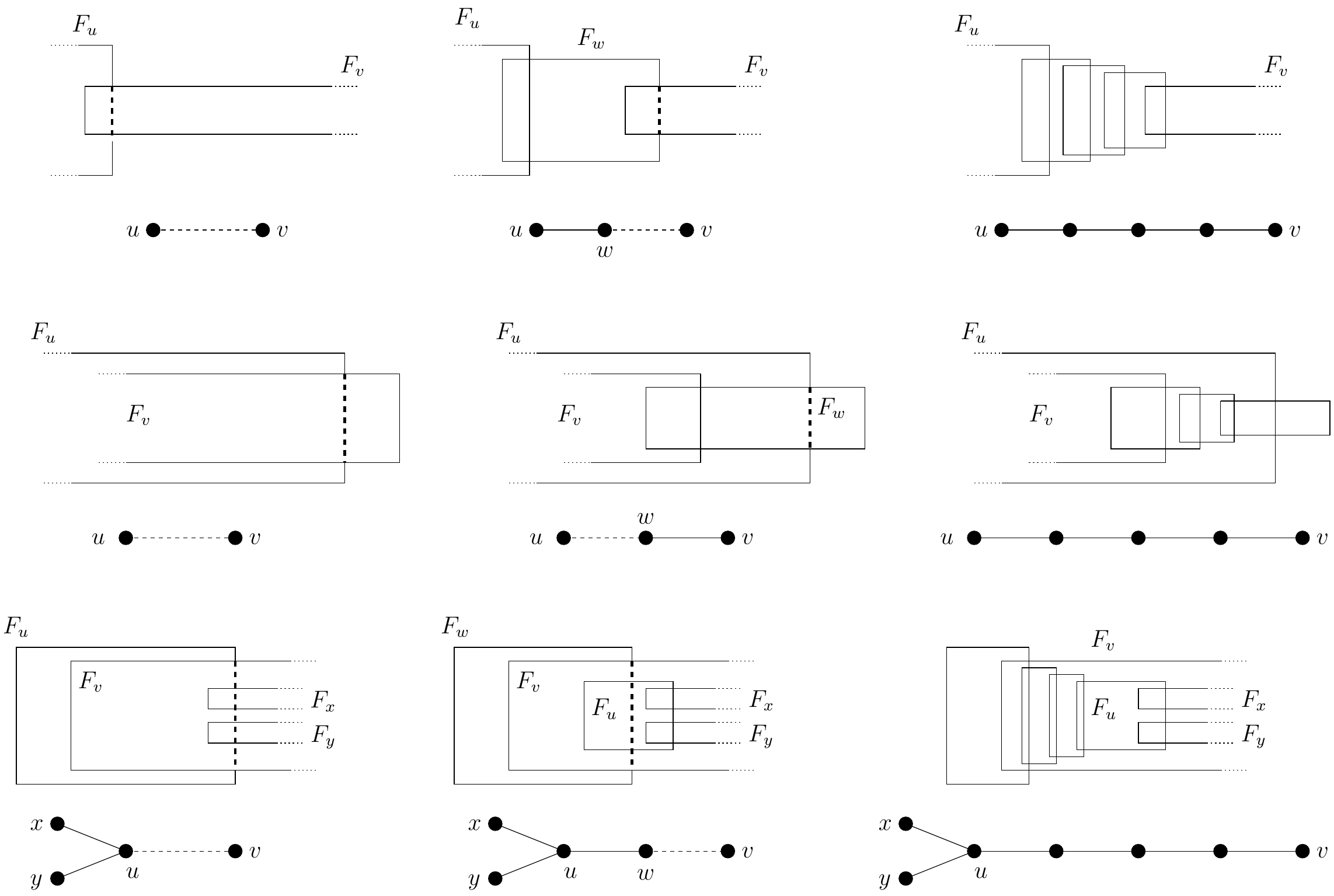}
\caption{A vertical dashed line corresponds to an intersection we
  replace by (a path of) frames, so that given a frame representation
  of the original graph, we can deduce a frame representation of any
  graph obtained by subdividing the edge
  $uv$.} \label{fig:pathinsertion}
\end{figure}

We can now state our result about triangle-free subdivisions of $K_4$.

\begin{theorem}
  Let $G$ be a triangle-free subdivision of $K_4$. Then $G$ is a \rbg s if and only if
  \begin{enumerate}
  \item
    $G$ is $K_4$ with at most 3 of the 6 edges subdivided, or
  \item
    $G$ is $K_4$ with 4 subdivided edges and the two non-subdivided
    edges share a vertex.
  \end{enumerate}
\end{theorem}

\begin{proof}
From Corollary~\ref{cor:noscs}, we know that any type 6 or type 5
subdivision of $K_4$, and any type 4 subdivision of $K_4$ in which the
non-subdivided edges of $K_4$ do not share a vertex cannot be a \rbg.

\begin{figure}[htbp]
\centering
\includegraphics[scale=0.6]{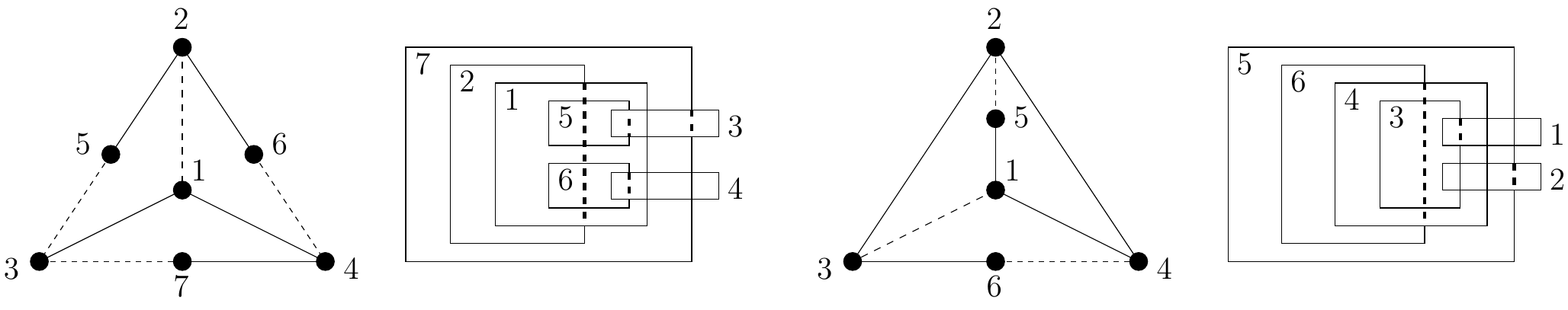}
\caption{Frame representations of triangle-free subdivisions of $K_4$
  that are \rbg s.} \label{fig:repk4}
\end{figure}

For the other triangle-free subdivisions of $K_4$, the construction is
given in Figure~\ref{fig:repk4}.  The convention on the figure is the
following: vertical dashed lines correspond to intersections we replace by
(a path of) frames according to Figure~\ref{fig:pathinsertion}.

For the graph on the left of Figure~\ref{fig:repk4}, we can subdivide
the edges $35$ and $46$ as depicted in Figure~\ref{fig:pathinsertion}
(top), the edge $37$ as depicted in Figure~\ref{fig:pathinsertion}
(middle), and the edge $12$ as depicted in
Figure~\ref{fig:pathinsertion} (bottom).

For the graph on the right of Figure~\ref{fig:repk4}, we can subdivide
the edge $13$ as depicted in Figure~\ref{fig:pathinsertion}
(top), the edge $25$ as depicted in Figure~\ref{fig:pathinsertion}
(middle), and the edge $46$ as depicted in
Figure~\ref{fig:pathinsertion} (bottom).

It can be checked that for these two graphs, the introduction of (paths
of) frames as depicted in Figure~\ref{fig:pathinsertion} yields
representations satisfying (1)--(4) in Definition~\ref{def:rbg},
therefore all the subdivisions of $K_4$ considered here are \rbg s.
\end{proof}

\section{The construction}\label{sec:cons}

Our ultimate goal is to characterize (multi)graphs $G$ such that all $\ge \! \! 2$ subdivisions
of $G$ appear as an induced subgraph in the construction
of~\cite{PKK13}. But in the previous sections, we instead
characterized (multi)graphs $G$ where all $\ge \! \! 2$ subdivisions of $G$ are \rbg s (Theorem \ref{thm:alg}), which
at first seems more restrictive. In this section, we bridge this gap
by showing that the two classes are in fact the same.

We first show that any graph appearing as an induced subgraph in the
construction of~\cite{PKK13} can be obtained by repeatedly applying
two fairly simple operations, \add\ and \join. We then deduce that \rbg s that are $\ge \! \!
2$-subdivisions of some multigraph appear as an induced subgraph in
the construction (Theorem~\ref{thm:construction}).

\begin{definition}
  A \emph{\gssp} $(G,\S)$ is a graph $G$ together with a set $\S$ of
  stable sets of $G$.

  %% A \gssp\ $(G,\S)$ is a \emph{sub\gssp} of $(H,\S')$ if $G$ is a
  %% subgraph of $H$ and $\S$ is a subset of $\S'$ restricted to the
  %% vertices of $G$.

  A \gssp\ $(G,\S)$ is an \emph{induced sub\gssp} of $(H,\S')$ if $G$
  is an induced subgraph of $H$ and $\S$ is a subset of the
  restriction of $\S'$ to the vertices of $G$.
\end{definition}

We define an iterative process which yields exactly the graphs in
Burling and Pawlik et al.'s construction.

\begin{definition}
  We define a procedure $\next$ which takes as input a \gssp\ $(G,\S)$
  and returns a \gssp\ $(G',\S')$. $(G',\S')$ is obtained from $(G,\S)$ by
  \begin{enumerate}
  \item
    adding $|\S|$ disjoint
    copies $(H_S,\S(H_S))$ of $(G,\S)$, indexed by stable sets $S\in
    \S$,
  \item adding a vertex $v_{S,T}$ whose neighborhood is exactly $T$ for each $S\in \S$ and for each $T \in \S(H_S)$, and
  \item setting $\S'$ as the union of $\{S\cup T \,|\, S\in \S,\, T
   \in \S(H_S) \}$ and $\{S\cup \{v_{S,T}\} \,|\, S\in \S,\, T \in
   \S(H_S) \}$.
  \end{enumerate}
\end{definition}

\begin{figure}[htbp]
\centering \includegraphics[scale=0.7]{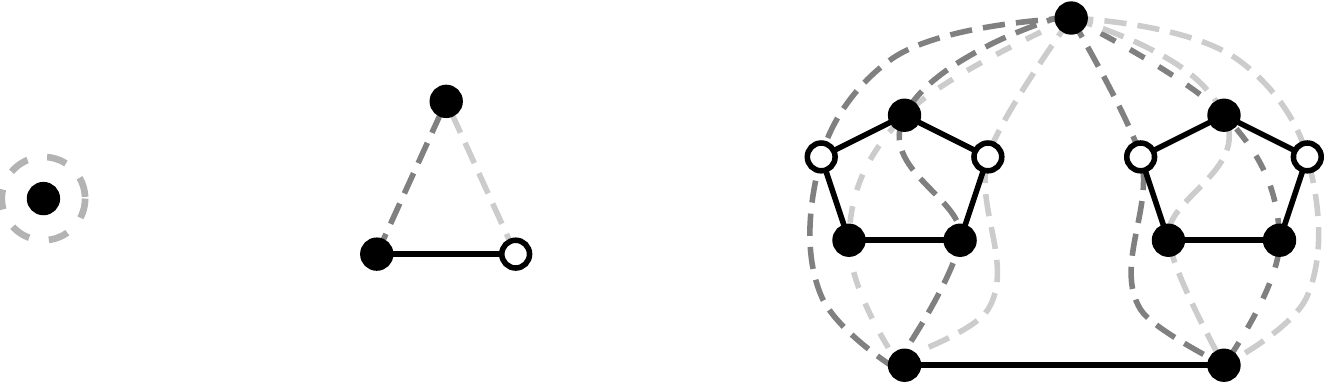}
\caption{The \gssp s from the first 3 steps of the construction. Vertices $v_{S,T}$ are
  white and each stable set in $S \in \S$ is a dashed line through the elements of $S$.} \label{fig:constr}
\end{figure}

\begin{definition}
  We say that a \gssp\ $(G,\S)$ is \emph{constructible} if it is an induced
  sub\gssp\ of $\next^i(G_0,\S_0)$ for some $i$, where $G_0 = K_1$,
  $\S_0 = \{V(G_0)\}$.
\end{definition}

We drew the first few \gssp s $\next^j(G_0,\S_0)$ in
Figure~\ref{fig:constr} and it is not too difficult to check that the
graph of $\next^j(G_0,\S_0)$ is the $j$th graph of Burling and Pawlik
et al.'s construction for each $j$. We can now state this section's
main result.

\begin{theorem}\label{thm:construction}
  For every \rbg\ $H$ that is also a $\ge \! \! 2$-subdivision of some
  multigraph, there is a subset $\S$ of stable sets of $H$ for which $(H,\S)$ is
  constructible.
\end{theorem}

As a consequence, any \rbg\ $H$ that is also a $\ge \! \! 2$-subdivision
(of some multigraph) appears as an induced subgraph of the
construction. 

To prove Theorem~\ref{thm:construction}, we use two simple operations
that preserve constructability rather than $\next$.

\begin{definition}
  By {\em adding a vertex $v$ on $S$} to a \gssp\ $(G,\S)$ with $S \in
  \S$, we mean to build a new \gssp\ $(H,\S')$ where $H$ is $G$ with a
  new vertex labelled $v$ whose neighborhood is $S$ and $\S'=\S \cup \{\{v\}\}$.

  We define the function $\add$ as $\add((G,\S), S) = (H,\S')$.

  By {\em joining} a \gssp\ $(G_1,\S_1)$ to a \gssp\ $(G_2,\S_2)$ on
  $S \in \S_2$, we mean to build a \gssp\ $(H,\S')$ where $H$ is the
  disjoint union of $G_1$ and $G_2$ and $\S' = (\S_2-\{S\}) \cup \{S \cup
  S_1\,|\, S_1 \in \S_1\}$.

  We define the function $\join$ as $\join((G_1,\S_1), (G_2,\S_2), S) = (H,\S')$.
\end{definition}

In other words, the operation $\add$ adds a vertex to the graph
adjacent to all vertices of a specific stable set and then adds a
stable set containing only this new vertex. The second operation is
the disjoint union of two graphs and the elements of a specific stable
set in the second \gssp\ are added to all stable sets of the first
\gssp.

To simplify the discussion, we allow joining on the empty stable set
(which results in the disjoint union of the graph and the disjoint
union of the stable sets). This can be simulated by adding a vertex $v$ using
$\add$, joining on $\{v\}$, and then removing $v$ (by taking an induced sub\gssp).

The following observations tell us that applying these two operations to
constructible \gssp s yields a constructible \gssp.

\begin{observation}
  $\add((G,\S), S)$ is an induced sub\gssp\ of $\next(G,\S)$ for any $S \in \S$.
\end{observation}

\begin{proof}
To see this, take any $T \in \S$ and note that $v_{T,S}$ has the
desired neighborhood in the subgraph of the graph of $\next(G,\S)$
induced by $V(G_T)\cup \{v_{T,S}\}$. 

Moreover by definition of $\next(G,\S)$, the stable sets in the
\gssp\ of $\next(G,\S)$ induced by $V(G_T)\cup \{v_{T,S}\}$ are
precisely $\S \cup \{v_{T,S}\}$, as desired.
\end{proof}

\begin{observation}
  If $(G_1,\S_1)$ and $(G_2,\S_2)$ are both induced sub\gssp s of
  $(H,\S)$, then for any $S\in \S_2$, $\join((G_1,\S_1), (G_2,\S_2), S)$ is an induced
  sub\gssp\ of $\next(H,\S)$.
\end{observation}

\begin{proof}
To see this, consider $\next(H,\S)$: the original copy of $(H,\S)$
contains an induced copy of $(G_2,\S_2)$, and the new copy
$(H_S,\S(H_S))$ contains an induced copy of $(G_1,\S_1)$. 

There exists $S' \in \S$ such that $S = S' \cap V(G_2)$ and for each
$S_1 \in \S_1$, there exists $S_1' \in \S$ such that $S_1 = S_1' \cap
V(G_1)$. Note that $S' \cup S_1'$ is a stable set of $\next(H,\S)$,
and that $S \cup S_1 = (S' \cup S_1') \cap (V(G_1) \cup V(G_2))$.
Moreover, for each $S_2 \in \S_2 - \{S\}$, there exists $S_2' \in \S$ such
that $S_2 = S_2' \cap V(G_2)$. Note that $S_2' \cup \{v_{S'_2,S'}\}$ is
a stable set of $\next(H,\S)$, and that $S_2 = (S_2' \cup
\{v_{S'_2,S'}\}) \cap (V(G_1) \cup V(G_2))$.

Consequently, $\join((G_1,\S_1), (G_2,\S_2), S)$ is an induced
  sub\gssp\ of $\next(H,\S)$.
\end{proof}

We sum up the previous observations into the following remark.

\begin{remark}
  If $(G_1,\S_1)$, $(G_2,\S_2)$ are constructible \gssp s, then for any
  $S \in \S_2$, $\add((G_2,\S_2), S)$ and $\join((G_1,\S_1), (G_2,\S_2),
  S)$ are constructible \gssp s.
\end{remark}

%We are now ready to prove Theorem \ref{thm:construction}. We prove it by induction and use the two following lemmas, one for the base case and one for the inductive case.
%It is an easy consequence of $\add$'s constructability preservation.

$\add$'s preservation of constructability has the following easy consequence.

\begin{lemma}\label{lem:tree}
  For any tree $T$, $(T,\{\{u\}\,|\,u\in T\})$ is a constructible \gssp.
\end{lemma}

\begin{proof}
  Start with the singleton \gssp\ and repeatedly apply $\add$ to build $T$.
\end{proof}

We will now need the following decomposition result, which is a
direct consequence of the proof of Theorem~\ref{thm:alg}. 

\begin{corollary}\label{cor:cons}
For any connected multigraph $G$ such that some $\ge \! \!
2$-subdivision $H$ of $G$ is a \rbg, $G$ has a rooted block decomposition
where the root block has a feedback vertex, and for each block $B$
distinct from the root block, the parent cut-vertex of $B$ is a
feedback vertex of $B$.
\end{corollary}

 This decomposition of $G$ induces a decomposition of $H$ with the
 same properties. We insist on the fact that this decomposition of $H$
 is not a block decomposition as defined earlier, since a block
 consisting of a single edge in $G$ corresponds to a block consisting
 of a path on at least 3 edges in $H$ (such a path is not
 2-connected). However, any block distinct from a single edge in $G$
 corresponds to a 2-connected block in $H$. We will refer to this
 decomposition of $H$ as a \emph{pseudo-decomposition}, and the $\ge
 \! \! 2$-subdivision of each block of $G$ will be called a
 \emph{pseudo-block} of $H$.

Note that this pseudo-decomposition has the additional property that any
cut-vertex of $G$ distinct from $r$ (the feedback vertex of the root
block) is at distance at least 3 from $r$ in $H$. In what follows, $r$
will be simply called \emph{the root} of $H$.

\smallskip

Our proof of Theorem~\ref{thm:construction} uses the following technical lemma which can be thought of as a strenghening of Theorem~\ref{thm:construction} that is better adapted for a proof by induction. 

\begin{lemma}\label{lem:inter}
Let $H$ be a connected \rbg\ that is a $\ge \! \! 2$-subdivision of some
multigraph $G$, and let $r$ be the root of $H$. Then
there is a set $\S$ of stable sets of $H- N[r]$ containing all singletons
$\{v\}$, where $v$ is at distance 2 from $r$ in $H$, such that the
\gssp\ $(H- N[r],\S)$ is constructible.
\end{lemma}

\begin{proof}
We prove the result by induction on the number of vertices of $H$. If
$H$ is 2-connected then $H-N[r]$ is a tree and the result follows
from Lemma~\ref{lem:tree}. Otherwise, for any cut-vertex $s$ of $G$
lying in the root block $R$, let $R,H_s^1,\ldots,H_s^k$ be the
subgraphs of $H$ induced by the vertex-set of each component of $H-s$
together with $s$. By induction, for each $s$ and each $i$, there
is a set $\S_s^i$ of stable sets of $H_s^i-N[s]$ including all
singletons $\{v\}$, where $v$ is at distance 2 from $s$ in $H_s^i$,
and such that $(H_s^i-N[s],\S_s^i)$ is constructible. Consequently, it
follows from the fact that the disjoint union of two constructible
\gssp s is a constructible \gssp \ (see the remark above on joining on
the empty set), that for any $s$, the disjoint union of all
$(H_s^i-N[s],\S_s^i)$, $i\ge 1$, forms a constructible \gssp. We will
refer to this \gssp\ as $(H_s,\S_s)$.

By Lemma~\ref{lem:tree}, $H_0=R-N[r]$ together with the set
$\S_0=\{\{v\}\,|\,v \in R-N[r] \}$ is a constructible \gssp. Let
$s_1,\ldots,s_\ell$ be the cut-vertices of $G$ lying in $R-r$. We
define two graphs $G_i$ and $H_i$ and a family of stable sets $\S_i$
iteratively as follows. For $i=1\ldots \ell$, let $(G_i,\S_i)$ be
obtained by joining $(H_{s_i},\S_{s_i})$ and $(H_{i-1},\S_{i-1})$ on
$\{s_i\}$. This \join\ operation creates stable sets $\{s_i,u\}$ for
all vertices $u$ at distance 2 from $s_i$ in $H_{s_i}$ (while $\S_i$
still contains singletons $\{s_j\}$ for any $j>i$). Let $H_i$ be
obtained from $G_i$ by adding, for each such pair $\{s_i,u\}$ a new
vertex adjacent to $s_i$ and $u$. We now define a \gssp\ $(H',\S')$ as
follows: if $r$ is not a cut-vertex in $G$, then
$(H',\S')=(H_\ell,\S_\ell)$, and otherwise $(H',\S')$ is the
\gssp\ obtained by taking the disjoint union of $(H_\ell,\S_\ell)$ and
$(H_r,\S_r)$ (i.e., joining them on the empty set).

It follows from the fact that $H$ is a $\ge \! \! 2$-subdivision of some
multigraph, that $H'$ is precisely $H-N[r]$ and $\S'$ has the desired
property (since vertices at distance two from $r$ in $H$ are not
cut-vertices).
\end{proof}

It remains to prove Theorem~\ref{thm:construction}.

\smallskip

\noindent \emph{Proof of Theorem~\ref{thm:construction}.} Observe that
$(H,\S)$ is constructible for some set $\S$ of stable sets of $H$ if
and only if all its connected components are. Hence, it is enough to
prove the theorem when $H$ is connected. In this case, by
Lemma~\ref{lem:inter}, there is a set $\S'$ of stable sets of
$H-N[r]$, where $r$ is the feedback vertex of the root pseudo-block of
the pseudo
decomposition, such that $\S'$ contains all singletons $\{v\}$ where
$v$ is at distance two from $r$ in $H$. Let $H_r$ consist of a single
vertex $r$, and $\S_r=\{\{r\}\}$.  We join the constructible
\gssp\ $(H-N[r],\S')$ and $(H_r,\S_r)$ on $\{r\}$, and the obtained
\gssp\ contains pairs $\{u,r\}$ for any $u$ is at distance two from
$r$ in $H$. The graph obtained by adding a vertex adjacent to $u$ and
$r$, for every such pair, is precisely $H$. \qed

\medskip

Combining Theorems~\ref{thm:alg}
and~\ref{thm:construction}, and the discussion about joining on the
empty set, we obtain the following immediate corollary.

\begin{corollary}
For any multigraph $G$, either all $\ge \! \! 2$-subdivisions of $G$
appear as induced subgraphs in the construction of Burling and Pawlik
et al., or none of them do.
\end{corollary}

\end{document}